\documentclass[a4paper,12pt,leqno]{article}
\usepackage[margin=2.5cm]{geometry}
\usepackage{amsmath}
\usepackage{amssymb}
\usepackage{amsxtra}
\usepackage{amsthm}
\usepackage[pagebackref,breaklinks]{hyperref}
\usepackage[T1]{fontenc}
\usepackage{lmodern}
\usepackage[lite,abbrev,msc-links]{amsrefs}
\usepackage{colonequals}
\usepackage{pict2e}

\newcommand{\toplabel}{9.4}
\newcommand{\bottomlabel}{-0.8}
\newcommand{\leftlabel}{-1}
\newcommand{\rightlabel}{9.1}
\newcommand{\End}{\mathrm{End}}
\newcommand{\Fr}{\mathrm{Fr}}
\newcommand{\Kl}{\mathrm{Kl}}
\newcommand{\Sp}{\mathrm{Sp}}
\newcommand{\SL}{\mathrm{SL}}
\newcommand{\GL}{\mathrm{GL}}
\newcommand{\SO}{\mathrm{SO}}
\newcommand{\rO}{\mathrm{O}}
\newcommand{\Res}{\mathrm{Res}}
\newcommand{\tr}{\mathrm{tr}}
\newcommand{\swan}{\mathrm{swan}}
\newcommand{\C}{\mathbf{C}}
\newcommand{\F}{\mathbf{F}}
\newcommand{\bone}{\mathbf{1}}
\newcommand{\cA}{\mathcal{A}}
\newcommand{\cK}{\mathcal{K}}
\newcommand{\cL}{\mathcal{L}}
\newcommand{\cX}{\mathcal{X}}
\newcommand{\Xbar}{\bar{\cX}}
\newcommand{\Klb}{\overline{\Kl}}
\newcommand{\Fqb}{\overline{\F_q}}
\newcommand{\Gm}{\mathbf{G}_{\mathrm{m}}}
\newcommand{\Gms}[1]{\mathbf{G}_{\mathrm{m},#1}}

\theoremstyle{definition}
\newtheorem{definition}[subsection]{Definition}
\newtheorem{remark}[subsection]{Remark}
\theoremstyle{plain}
\newtheorem{lemma}[subsection]{Lemma}
\newtheorem{theorem}[subsection]{Theorem}
\newtheorem{cor}[subsection]{Corollary}

\begin{document}
\numberwithin{equation}{section}
\title{On the distribution of Jacobi sums}
\author{Qing Lu\thanks{School of Mathematical Sciences, University of Chinese Academy of
Sciences, Beijing 100049, China; Academy of Mathematics and Systems Science,
Chinese Academy of Sciences, Beijing 100190, China; email:
\texttt{qlu@ucas.ac.cn}. Partially supported by National Natural Science
Foundation of China Grant 11371043.}\and Weizhe Zheng\thanks{Morningside
Center of Mathematics, Academy of Mathematics and Systems Science, Chinese
Academy of Sciences, Beijing 100190, China; email:
\texttt{wzheng@math.ac.cn}. Partially supported by China's Recruitment
Program of Global Experts; National Natural Science Foundation of China
Grant 11321101; Hua Loo-Keng Key Laboratory of Mathematics and National
Center for Mathematics and Interdisciplinary Sciences, Chinese Academy of
Sciences.}\and Zhiyong Zheng\thanks{School of Mathematics and Systems
Science, Beihang University, Beijing 100191, China; email:
\texttt{zhengzhiyong@buaa.edu.cn}. Partially supported by Program 863 Grant
2013AA013702; Program 973 Grant 2013CB834205.} }
\date{}
\maketitle

\begin{abstract}
Let $\F_q$ be a finite field of $q$ elements. For multiplicative
characters $\chi_1,\dots, \chi_m$ of $\F_q^\times$, we let
$J(\chi_1,\dots, \chi_m)$ denote the Jacobi sum. Nicholas Katz and Zhiyong
Zheng showed that for $m=2$, the normalized Jacobi sum
$q^{-1/2}J(\chi_1,\chi_2)$ ($\chi_1\chi_2$ nontrivial) is asymptotically
equidistributed on the unit circle as $q\to \infty$, when $\chi_1$ and
$\chi_2$ run through all nontrivial multiplicative characters of
$\F_q^\times$. In this paper, we show a similar property for $m\ge 2$.
More generally, we show that the normalized Jacobi sum
$q^{-(m-1)/2}J(\chi_1,\dots,\chi_m)$ ($\chi_1\dotsm \chi_m$ nontrivial) is
asymptotically equidistributed on the unit circle, when $\chi_1,\dots,
\chi_m$ run through arbitrary sets of nontrivial multiplicative characters
of $\F_q^\times$ with two of the sets being sufficiently large. The case
$m=2$ answers a question of Shparlinski.
\end{abstract}

\section{Introduction}
Let $\F_q$ be a finite field of characteristic $p$ with $q$ elements, and
let $\C$ be the field of complex numbers. We let $\Psi$ denote the set of
nontrivial additive characters $\F_q \to \C^\times$. We let $\Xbar$ (resp.\
$\cX$) denote the set of multiplicative characters (resp.\ nontrivial
multiplicative characters) $\F_q^\times\to \C^\times$. For $\psi\in \Psi$
and $\chi\in \Xbar$, we consider the Gauss sum
\[G(\psi,\chi)=\sum_{a\in \F_q^\times} \psi(a)\chi(a).\]
For $m\ge 2$, $\chi_1,\dots,\chi_m\in \Xbar$, we consider the Jacobi sum
\[J(\chi_1,\dots,\chi_m)=\sum_{\substack{a_1,\dots,a_m\in \F_q^\times\\
a_1+\dots+a_m=1}}\chi_1(a_1)\dotsm\chi_m(a_m).
\]
It is known that for $\chi,\chi_1,\dots,\chi_m\in \cX$,
$\chi_1\dotsm\chi_m\neq \bone$, where $\bone$ denotes the trivial
multiplicative character,
\[\lvert G(\psi,\chi)\rvert =q^{1/2}, \qquad \lvert J(\chi_1,\dots,\chi_m)\rvert = q^{(m-1)/2}.\]

Nicholas Katz and Zhiyong Zheng showed in \cite[Theorem 1]{KZ} that the
normalized Gauss sums
\[\{q^{-1/2}G(\psi,\chi)\}_{\psi\in \Psi,\ \chi\in \cX}\]
and, for $m=2$, the normalized Jacobi sums
\[\{q^{-1/2}J(\chi_1,\chi_2)\}_{\chi_1,\chi_2\in \cX,\ \chi_1\chi_2\neq \bone}\]
are asymptotically equidistributed in the unit circle as $q\to \infty$.
Shparlinski showed in \cite{Shp} that the normalized Gauss sums
\[\{q^{-1/2}G(\psi,\chi)\}_{\psi\in \Phi,\ \chi\in \cA},\]
where $\psi$ and $\chi$ run through arbitrary subsets $\Phi\subseteq \Psi$
and $\cA\subseteq \cX$ satisfying $\#\Phi\#\cA\ge q^{1+\epsilon}$ for a
constant $\epsilon> 0$, are asymptotically equidistributed in the unit
circle as $q\to \infty$, and asked whether a similar property holds for
$q^{-1/2}J(\chi_1,\chi_2)$.

The goal of this paper is to study more generally equidistribution
properties of the normalized Jacobi sums
\begin{equation}\label{e.seq}
\{q^{-(m-1)/2}J(\chi_1,\dots,\chi_m)\}_{\chi_i\in \cA_i,\ \chi_1\dotsm \chi_m\neq \bone},
\end{equation}
for $m\ge 2$, where the $\chi_i$'s run through arbitrary nonempty subsets
$\cA_i\subseteq \cX$, $i=1,\dots,m$. We show that \eqref{e.seq} is
asymptotically equidistributed in the unit circle when two of the subsets
are sufficiently large in the sense that $q\ln^2 q/\max_{i\neq
j}\#\cA_i\#\cA_j\to 0$. The case $m=2$ gives an affirmative answer to
Shparlinski's question. Moreover, we give better equidistribution estimates
when some (or all) of the subsets are $\cX$. As in \cite{KZ} and \cite{Shp},
we do not restrict the way how $q$ approaches infinity. In particular, we do
not fix $p$.

To formulate our results, we need the following notion.

\begin{definition}\label{d.D}
The \emph{discrepancy} of a finite multiset of complex numbers
$\{z_1,\dots,z_N\}$ on the unit circle is defined to be
\[D=\sup_{a\le b \le a+1} \left\lvert \frac{T(a,b)}{N}-(b-a)\right\rvert,\]
where $T(a,b)$ is the number of $1\le i \le N$ such that there exists
$c\in[a,b]$ satisfying $z_i=e^{2\pi i c}$.  For $N=0$ we put $D=1$. We say
that a sequence or net of such multisets
$(\{z_{\alpha,1},\dots,z_{\alpha,N_\alpha}\})_{\alpha\in I}$ is
\emph{asymptotically equidistributed} if $D=o(1)$.
\end{definition}

For $N\ge 1$ we have $\frac{1}{N}\le D\le 1$. We let $D(\cA_1,\dots,\cA_m)$
denote the discrepancy of the multiset \eqref{e.seq}.

\begin{theorem}\label{t.1}
Let $m\ge 2$ and let $\cA_1,\dots,\cA_m$ be nonempty subsets of $\cX$. Let
$A_1=\#\cA_1$, $A_2=\#\cA_2$. Then
\begin{gather}
\label{e.11} D(\cA_1,\dots,\cA_m)\le 3 A_1^{-1/3}q^{1/6}+\tfrac{1}{9}(A_1A_2)^{-1/2}q^{1/2}(6+\ln
q),\\
\label{e.12} D(\cA_1,\dots,\cA_m)\le 2A_1^{-2/7}A_2^{-1/7}q^{3/14}+\tfrac{1}{5}A_1^{-1/2}A_2^{-1/4}q^{1/2}(4+\ln q).
\end{gather}
\end{theorem}

Since $D(\cA_1,\dots,\cA_m)$ is symmetric in the $\cA_i$'s, \eqref{e.11} is
equivalent to
\[D(\cA_1,\dots,\cA_m)\le 3(\max_i \#\cA_i)^{-1/3}q^{1/6}+\tfrac{1}{9}(\max_{i\neq j} \#\cA_i\#\cA_j)^{-1/2}q^{1/2}(6+\ln
q).
\]
Therefore, \eqref{e.seq} is asymptotically equidistributed when $q\ln^2
q/\max_{i\neq j}\#\cA_i\#\cA_j\to 0$. We note that this condition cannot be
substantially improved. In fact, for $\cA_2,\dots,\cA_m$ satisfying
$\#\cA_2=\dots=\#\cA_m=1$, there exists $\cA_1$ satisfying $\#\cA_1\ge
(q-3)/2$ such that \eqref{e.seq} is contained in a semicircle, so that
$D(\cA_1,\dots,\cA_m)\ge \frac{1}{2}$.

\begin{cor}\label{c.1}
There exists a constant $C$ such that for all $m\ge 2$ and for nonempty
subsets $\cA_1,\dots,\cA_m$ of $\cX$, we have
\[D(\cA_1,\dots,\cA_m)\le C q^{-f(\log_q \#\cA_1,\log_q \#\cA_2)}\ln q,\]
where $f\colon [0,1]\times [0,1]\to [0,\frac{3}{14}]$ is the function
satisfying $f(x,y)=f(y,x)$ and such that for $x\ge y$,
\begin{equation}\label{e.f}
f(x,y)=\begin{cases}
0& x+y\le 1,\\
\tfrac{1}{2}x+\tfrac{1}{2}y-\tfrac{1}{2}& x+y\ge 1\text{ and }x+3y\le 2,\\
\tfrac{1}{3}x-\tfrac{1}{6} & x+3y\ge 2\text{ and }2x+3y\le 4,\\
\tfrac{1}{2}x+\tfrac{1}{4}y-\tfrac{1}{2}&2x+3y\ge 4\text{ and }2x+y\le \tfrac{8}{3},\\
\tfrac{2}{7}x+\tfrac{1}{7}y-\tfrac{3}{14}&2x+y\ge \tfrac{8}{3}.
\end{cases}
\end{equation}
\end{cor}

Note that the function $f(x,y)$ is nondecreasing with respect to both $x$
and $y$, continuous, and is linear on each piece of the following partition
of $[0,1]\times[0,1]$
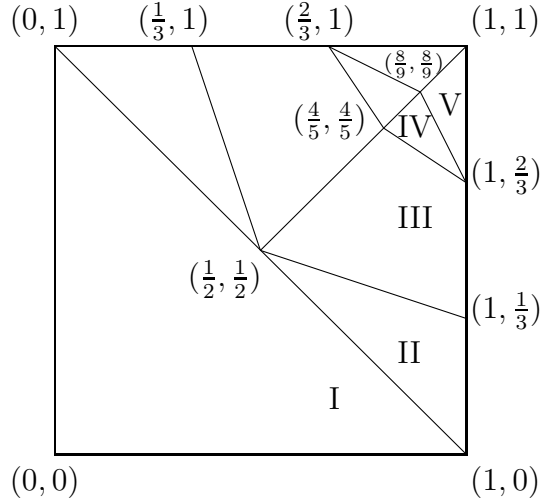
\begin{center}
\begin{picture}(14,11)(-2,-1)
  \put(0,0){\framebox(9,9){}}
  \put(4.5,4.5){\line(1,1){4.5}}
  \put(4.5,4.5){\line(3,-1){4.5}}
  \put(4.5,4.5){\line(-1,3){1.5}}
  \put(7.2,7.2){\line(3,-2){1.8}}
  \put(7.2,7.2){\line(-2,3){1.2}}
  \put(0,9){\line(1,-1){9}}
  \put(8,8){\line(1,-2){1}}
  \put(8,8){\line(-2,1){2}}
  \put(\leftlabel,\bottomlabel){$(0,0)$}
  \put(\rightlabel,\bottomlabel){$(1,0)$}
  \put(\rightlabel,\toplabel){$(1,1)$}
  \put(\leftlabel,\toplabel){$(0,1)$}
  \put(2.9,3.7){$(\frac{1}{2},\frac{1}{2})$}
  \put(\rightlabel,6){$(1,\frac{2}{3})$}
  \put(5,\toplabel){$(\frac{2}{3},1)$}
  \put(1.8,\toplabel){$(\frac{1}{3},1)$}
  \put(\rightlabel,3){$(1,\frac{1}{3})$}
  \put(5.2,7.2){$(\frac{4}{5},\frac{4}{5})$}
  \put(7.2,8.5){\scriptsize $(\frac{8}{9},\frac{8}{9})$}
  \put(6,1){I}
  \put(7.5,2){II}
  \put(7.5,5){III}
  \put(7.5,7){IV}
  \put(8.4,7.5){V}
\end{picture}
\end{center}
with $f(0,0)=f(1,0)=f(\frac{1}{2},\frac{1}{2})=0$,
$f(\frac{4}{5},\frac{4}{5})=\frac{1}{10}$,
$f(1,\frac{1}{3})=f(1,\frac{2}{3})=f(\frac{8}{9},\frac{8}{9})=\frac{1}{6}$,
$f(1,1)=\frac{3}{14}$. The pieces marked with I, II, III, IV, V correspond
to the five cases of \eqref{e.f}.

Next we give better upper bounds for the discrepancy when some of the
subsets are actually $\cX$. We put
$D_k(\cA_1,\dots,\cA_m)=D(\cA_1,\dots,\cA_m,\cX,\dots,\cX)$ for $m,k\ge 1$
and $D_k=D(\cX,\dots,\cX)$ for $k\ge 2$, where $\cX$ is repeated $k$ times.

\begin{theorem}\label{t.2}
Let $m\ge 1$ and let $\cA_1,\dots,\cA_m$ be nonempty subsets of $\cX$. Let
$A_1=\#\cA_1$. Then, for $k\ge 2$,
\begin{gather}
\label{e.21}D_k(\cA_1,\dots,\cA_m)\le 2q^{-\frac{k}{2(k+1)}}(1+k!q^{-1/6}\ln q),\\
\label{e.22}D_k(\cA_1,\dots,\cA_m)\le 2A_1^{-\frac{1}{2k+3}}q^{-\frac{2k-1}{2(2k+3)}}\{1+q^{-2/7}[7^{k-1} +(2k+1)!!^{1/2}\ln q]\}.
\end{gather}
For $k=1$, we have
\begin{equation}
\label{e.23}D_1(\cA_1,\dots,\cA_m)\le 2
q^{-1/4}+\tfrac{1}{6}\delta A_1^{-1}(5+\ln q)(1+2q^{-1/2}),
\end{equation}
where $\delta=0$ if $m=1$ and $\delta=1$ if $m>1$. Moreover, for $A_1\ge
q^{3/4}$, we have
\begin{equation}\label{e.24}
D_1(\cA_1,\dots,\cA_m)\le 2A_1^{-1/5}q^{-1/10}(1+q^{-1/8}\ln q).
\end{equation}
\end{theorem}

\begin{cor}\label{c.2}
Let $k\ge 1$. There exists a constant $C_k$ such that for all $m\ge 1$
(assuming $m=1$ if $k=1$) and for nonempty subsets $\cA_1,\dots,\cA_m$ of
$\cX$, we have
\[D_k(\cA_1,\dots,\cA_m)\le C_k q^{-g_{k}(\log_q \#\cA_1)},\]
where $g_{k}\colon [0,1]\to [\tfrac{k}{2(k+1)},\tfrac{2k+1}{2(2k+3)}]$ is
the function
\[g_{k}(x)=\begin{cases}
  \tfrac{k}{2(k+1)}&x\le \tfrac{2k+1}{2k+2},\\
  \tfrac{1}{2k+3}x+\tfrac{2k-1}{2(2k+3)}&x\ge \tfrac{2k+1}{2k+2}.
\end{cases}
\]
For $k=1$, there exists a constant $C'$ such that for all $m\ge 1$, we have
\[D_1(\cA_1,\dots,\cA_m)\le C'q^{-h(\log_q \#\cA_1)}\ln q,\]
where $h\colon [0,1]\to [0,\frac{3}{10}]$ is the function
\[h(x)=\begin{cases}
x&0\le x\le \tfrac{1}{4},\\
\tfrac{1}{4}& \tfrac{1}{4}\le x\le \tfrac{3}{4}, \\
\tfrac{1}{5}x+\tfrac{1}{10}&\tfrac{3}{4}\le x \le 1.
\end{cases}
\]
\end{cor}

Note that the functions $g_{k}(x)$ and $h(x)$ are nondecreasing, continuous
and piecewise-linear. Corollary \ref{c.2} for $k=1$ improves the case
$\cA_1=\cX$ of Corollary \ref{c.1}, since $f(1,x)\le h(x)\le g_{1}(x)$.
Moreover, Corollary \ref{c.2} for $k\ge 2$ improves the case $\cA_1=\cX$ of
Corollary \ref{c.2} for $k-1$, since $g_{k-1}(1)=\tfrac{2k-1}{2(2k+1)}<
\tfrac{k}{2(k+1)}=g_{k}(0)\le g_{k}(x)$.

When all of the subsets are $\cX$, we have the following extension of
\eqref{e.21}.

\begin{theorem}\label{t.3}
For $k\ge 2$, $q\ge 3$, we have
\begin{equation}\label{e.3}
D_k\le 2q^{-\frac{k}{2(k+1)}}(1+k!q^{-1/6}\ln q).
\end{equation}
\end{theorem}

This improves the case $m=1$, $\cA_1=\cX$ of Corollary \ref{c.2} for $k-1$.
For $k=2$, we recover the result $D_2=O(q^{-1/3})$ of Katz and Zhiyong Zheng
\cite[Theorem~1]{KZ}.

To prove the above theorems, we use the Erd\H os-Tur\'an inequality together
with estimates of moments of Jacobi sums. Our method of estimating moments
of Jacobi sums is based on the theory of Kloosterman sheaves as in
\cite{KZ}, but we need estimates for higher tensor powers of Kloosterman
sheaves, which we give in Section~\ref{s.2}. We give estimates for moments
of Jacobi sums in Section~\ref{s.3}. In Section~\ref{s.4}, we prove the
upper bounds for the discrepancy and give a lower bound for $D_k$, $k\ge 3$.

\section{A key lemma}\label{s.2}
In the rest of this paper, we fix a nontrivial additive character $\psi$ on
$\F_q$ and omit it from the notation. For $n\ge 1$ and $a\in \F_q^\times$,
we consider the Kloosterman sum
\[\Kl_n(a)=\sum_{\substack{a_1,\dots,a_n\in \F_q^\times\\a_1\dotsm a_n=a}} \psi(a_1+\dots+a_n).\]
We have $\Kl_1(a)=\psi(a)$. The Fourier transform of $\Kl_n(a)$ is the
$n$-th power of the Gauss sum $G(\chi)$:
\begin{equation}\label{e.Fourier}
G(\chi)^n=\sum_{a\in \F_q^\times}\Kl_n(a)\chi(a)
\end{equation}
for all $\chi\in \Xbar=\widehat{\F_q^\times}$ \cite[4.0, page 47]{Katz}.

\begin{lemma}\label{l.Kl}
Let $n\ge 1$, $k,l\ge 0$. Let $\chi$ be a nontrivial multiplicative
character of $\F_q^\times$. Then
\begin{gather}
\label{e.Kl1}\left\lvert \sum_{a\in \F_q^\times} \Kl_n(a)^k\Klb_n(a)^l  -R q^\frac{(n-1)(k+l)+2}{2}\right\rvert\le
\left(\left\lfloor n^{k+l-1}-\frac{R}{n}\right\rfloor+R\right)q^{\frac{(n-1)(k+l)+1}{2}},\\
\label{e.Kl2}\left\lvert \sum_{a\in \F_q^\times} \chi(a)\Kl_n(a)^k\Klb_n(a)^l \right\rvert
\le \left\lfloor n^{k+l-1}-\frac{R}{n}\right\rfloor q^{\frac{(n-1)(k+l)+1}{2}}.
\end{gather}
Here $R=R^{k,l}_{p,n}$ is the dimension of $(V^{\otimes l}\otimes
(V^*)^{\otimes k})^G$, where $V$ is the standard complex representation of
$G$ of dimension $n$,
\[G=\begin{cases}
  \mu_p &n=1,\\
  \Sp_n& n\text{ even},\\
  \SL_n& p,n\ge 3\text{ odd},\\
  \SO_n& p=2,\ n\neq 1,7\text{ odd},\\
  G_2& p=2,\ n=7,
\end{cases}\]
and $\mu_p$ is the group of $p$-th roots of unity in $\C$.
\end{lemma}

Let $E\subset \C$ be a number field containing the $p$-th roots of unity and
let $\lambda$ be a finite place of $E$ not dividing $p$. Recall from Deligne
\cite[Th\'eor\`eme 7.8]{Deligne} that the Kloosterman sheaf $\cK_n$ is a
lisse $E_\lambda$-sheaf on $\Gms{\F_q}$ of rank $n$ and weight $n-1$
satisfying
\[\tr(\Fr_a,(\cK_n)_{\bar a})=(-1)^{n-1} \Kl_n(a),\]
where $\Fr_a$ is the geometric Frobenius at $a\in \Gm(\F_q)=\F_q^\times$ and
$\bar a$ is a geometric point above $a$. Moreover,
\[\tr(\Fr_a,(\cK_n\spcheck)_{\bar a})=(-1)^{n-1} q^{-(n-1)}\Klb_n(a).\]
The group $G$ in the lemma is the Zariski closure of the geometric monodromy
group of $\cK_n$ as computed by Katz \cite[Theorem 11.1]{Katz}.

Deligne's bound $\lvert \Kl_n(a)\rvert \le nq^{\frac{n-1}{2}}$ implies that
the left hand side of \eqref{e.Kl2} is bounded by
$n^{k+l}(q-1)q^{\frac{(n-1)(k+l)}{2}}$. Thus \eqref{e.Kl2} is nontrivial. We
will see in Remark \ref{r.R} that $R\le (k+l-1)!$ (by convention $(-1)!=1$),
so that \eqref{e.Kl1} provides a nontrivial upper bound for $\left\lvert
\sum_{a\in \F_q^\times} \Kl_n(a)^k\Klb_n(a)^l\right\rvert$ at least when $n$
is large relative to $k$ and $l$. For $k=2$, $l=1$, \eqref{e.Kl1} recovers
the bound $\left\lvert \sum_{a\in \F_q^\times}
\Kl_n(a)^2\Klb_n(a)\right\rvert\le
Rq^{\frac{3n-1}{2}}+n^2q^{\frac{3n-2}{2}}$ in \cite[Key Lemma 8, page
549]{KZ} (in this case $R=0$ or $1$, see Remark \ref{r.small} below).

\begin{proof}[Proof of Lemma \ref{l.Kl}]
Recall \cite[Th\'eor\`eme 7.8]{Deligne} that the local monodromy of $\cK_n$
at $0$ is unipotent and tame. The local monodromy at $\infty$ is totally
wild with Swan conductor $\swan_{\infty}(\cK_n)=1$, so that all breaks are
$1/n$ \cite[Lemma 1.11]{Katz}.

By the Grothendieck trace formula,
\[\sum_{a\in \F_q^\times} \Kl_n(a)^k\Klb_n(a)^l=(-1)^{(n-1)(k+l)}q^{(n-1)l}\sum_{i=0}^2(-1)^i\tr(\Fr_q,H^i_c),\]
where $H^i_c=H^i_c(\Gms{\Fqb},\cK_n^{\otimes k}\otimes
(\cK_n\spcheck)^{\otimes l})$. We have $H^0_c=0$ and, by Poincar\'e duality,
\[H^2_c\simeq
H^0(\Gms{\Fqb},\cK_n^{\otimes l}\otimes (\cK_n\spcheck)^{\otimes
k})\spcheck(-1)
\]
has dimension $h^2_c=R$. By \cite[Corollary 11.3]{Katz}, the arithmetic
fundamental group of $\cK_n(\frac{n-1}{2})$ (well-defined up to adjoining
$q^{\frac{n-1}{2}}$ to $E$) coincides with $G$. Thus
\[\tr(\Fr_q,H^2_c)=Rq^{\frac{(n-1)(k-l)+2}{2}}.\]
Moreover, $(n-1)(k+l)$ is even whenever $R>0$. By Deligne's Weil II
\cite[Th\'eor\`eme 3.3.1]{WeilII},
\[\left\lvert \tr(\Fr_q,H^1_c) \right\rvert \le h^1_c q^{\frac{(n-1)(k-l)+1}{2}},\]
where $h^1_c=\dim H^1_c$. The sheaf $\cK_n^{\otimes k}\otimes
(\cK_n\spcheck)^{\otimes l}$ has rank $n^{k+l}$ and is tame at $0$. All
breaks at $\infty$ of this sheaf are at most $1/n$ by \cite[Lemma 1.3]{Katz}
and at least $R$ breaks are $0$. It follows that the Swan conductor
\[\swan_\infty(\cK_n^{\otimes k}\otimes (\cK_n\spcheck)^{\otimes l})\le
\lfloor(n^{k+l}-R)/n\rfloor .
\]
The inequality \eqref{e.Kl1} then follows from the
Grothendieck-Ogg-Shafarevich formula \cite[Th\'eor\`eme  7.1]{GOS}
\[h^1_c=h^2_c+\swan_\infty(\cK_n^{\otimes k}\otimes (\cK_n\spcheck)^{\otimes l}).\]

For \eqref{e.Kl2}, we may assume that $E$ contains the image of $\chi$. Let
$\cL_\chi$ be the lisse $E_\lambda$-sheaf of rank $1$ on $\Gms{\F_q}$
corresponding to $\chi$. As the local monodromy at $0$ of $\cK_n^{\otimes
l}\otimes (\cK_n\spcheck)^{\otimes k}\otimes\cL_\chi\spcheck$ is given by a
successive extension of $\bar\chi$, we have
\[H^2_c(\Gms{\Fqb},\cL_\chi\otimes\cK_n^{\otimes k}\otimes (\cK_n\spcheck)^{\otimes l})\simeq H^0(\Gms{\Fqb},\cK_n^{\otimes l}\otimes (\cK_n\spcheck)^{\otimes k}\otimes\cL_\chi\spcheck)\spcheck(-1)=0.\]
The rest of the proof is completely similar to the proof of the first
assertion.
\end{proof}

\begin{remark}
We gather some formulas and bounds for the constant $R=R^{k,l}=R^{k,l}_{G}$
in the above lemma. We have $R^{k,l}=R^{l,k}$. For $G=\Sp_n$, $\SO_n$, or
$G_2$, $V^*\simeq V$ so that $R^{k,l}$ depends only on $k+l$ (and $G$). In
this case, we put $R^{k+l}=R^{k,l}$.

For $G=\mu_p$, $R^{k,l}=1$ if $k\equiv l \pmod p$ and $R^{k,l}=0$ otherwise.

For $G=\Sp_n$ ($n$ even), we let $V_\lambda$ denote the irreducible
representation corresponding to a partition $\lambda=(\lambda_1\ge \dots \ge
\lambda_{n/2}\ge 0)$ (where the $\lambda_j$'s are integers). We have
$V\simeq V_\sigma$, where $\sigma=(1,0,\dots,0)$. By King's formula
\cite[(4.14), (4.15), (4.31)]{King}, we have
\[V_\lambda\otimes V_\sigma\simeq \bigoplus_{\lambda'}V_{\lambda'},\]
where $\lambda'$ runs through $\sigma$-expansions and $\sigma$-contractions
of $\lambda$. Here we say that $\lambda'$ is a $\sigma$-expansion of
$\lambda$, or equivalently $\lambda$ is a $\sigma$-contraction of
$\lambda'$, if there exists $j$ satisfying $\lambda'_j=\lambda_j+1$ and
$\lambda'_{j'}=\lambda_{j'}$ for all $j'\neq j$. Thus $R^k$ is the number of
sequences of partitions $\lambda^{(0)},\dots, \lambda^{(k)}$ with
$\lambda^{(0)}=\lambda^{(k)}=(0,\dots,0)$, such that for each $0\le i<k$,
$\lambda^{(i+1)}$ is a $\sigma$-expansion or a $\sigma$-contraction of
$\lambda^{(i)}$.  Moreover, by classical invariant theory \cite[Section
VI.7]{Weyl}, $(V^{\otimes k})^G$ is spanned by the invariants given by
partitions of $\{1,\dots, k\}$ into pairs, so that $R^k\le (k-1)!!$, and
equality holds if and only if $k\le n$ and $k$ even. Here we adopt the
convention that $(-1)!!=1$. For $k$ odd, $R^k=0$.

For $G=\SO_n$ ($n$ odd), we let $V_\lambda$ denote the irreducible
representation of $\rO_n=\SO_n\times \{\pm 1\}$ corresponding to a partition
$\lambda=(\lambda_1\ge \dots ge\lambda_{n}\ge 0)$ satisfying
$\lambda^T_1+\lambda^T_2\le n$, where $\lambda^T$ denotes the conjugate of
$\lambda$. We have $V\simeq \Res^{\rO_n}_{\SO_n}V_\sigma$ and, for
$\lambda\neq \lambda'$, $\Res^{\rO_n}_{\SO_n} V_{\lambda}\simeq
\Res^{\rO_n}_{\SO_n} V_{\lambda'}$ if and only if
$\lambda^T_1+\lambda'^T_1=n$ and $\lambda^T_j=\lambda'^T_j$ for all $j>1$.
By King's formula for $\rO_n$ \cite[(4.14), (4.15)]{King}, we have
$V_\lambda\otimes V_\sigma\simeq \bigoplus_{\lambda'}V_{\lambda'}$, where
$\lambda'$ runs through $\sigma$-expansions and $\sigma$-contractions of
$\lambda$. Thus, for $k$ odd (resp.\ even), $R^k$ is the number of sequences
of partitions $\lambda^{(0)},\dots, \lambda^{(k)}$, where
$\lambda^{(i)}=(\lambda^{(i)}_1\ge \dots \ge \lambda^{(i)}_{n}\ge 0)$,
$\lambda^{(0)}=(0,\dots,0)$, $\lambda^{(k)}=(1,\dots,1)$ (resp.\
$\lambda^{(k)}=(0,\dots,0)$), such that for each $i$, $\lambda^{(i+1)}$ is a
$\sigma$-expansion or a $\sigma$-contraction of $\lambda^{(i)}$. Moreover,
by classical invariant theory \cite[Sections II.9, II.17]{Weyl}, for $k$
odd, $(V^{\otimes k})^G$ is spanned by the images of $\C\simeq \wedge^n V
\subset V^{\otimes n}$ under the expansion operators $V^{\otimes n}\to
V^{\otimes k}$ given by an injection $\{1,\dots, n\}\hookrightarrow
\{1,\dots, k\}$ and a partition of the complement into pairs, so that
$R^k=0$ for $k<n$ and $R^k\le \binom{k}{n}(k-n-1)!!\le (k-1)!$ for $k\ge n$
(assuming $n\ge 3$). For $k$ even, $(V^{\otimes k})^G$ is spanned by the
invariants given by partitions of $\{1,\dots, k\}$ into pairs, so that
$R^k\le (k-1)!!$, and equality holds if and only if $k\le 2n$.

For $G=G_2$, we let $V_\lambda$ denote the irreducible representation
corresponding to a partition $\lambda=(\lambda_1\ge \lambda_2\ge 0)$, so
that $V_{0,0}=\C$, $V_{1,0}=V$. By Littelmann's generalized
Littlewood-Richardson rule \cite[3.8]{Littel}, we have $V_\lambda\otimes
V^{\otimes k} \simeq \bigoplus_{\lambda'} V_{\lambda'}$, where $\lambda'$
satisfies one of the following
\begin{itemize}
\item $\lambda'$ is a $\sigma$-expansion or a $\sigma$-contraction of
    $\lambda$; or

\item $\lambda'_1=\lambda_1\pm 1$ and $\lambda'_2=\lambda_2\mp 1$; or

\item $\lambda'=\lambda$ and $\lambda_1>\lambda_2$.
\end{itemize}
Note that, for each $\lambda$, there are at most $7$ possibilities for
$\lambda'$. We have
\begin{gather*}
V^{\otimes 2}\simeq V_{0,0} \oplus V_{1,0} \oplus V_{2,0} \oplus V_{1,1}, \qquad
V^{\otimes 3}\simeq V_{0,0} \oplus V_{1,0}^{\oplus 4} \oplus V_{2,0}^{\oplus 3} \oplus V_{3,0} \oplus V_{1,1}^{\oplus 2} \oplus V_{2,1}^{\oplus 2},\\
V^{\otimes 4}\simeq V_{0,0}^{\oplus 4} \oplus V_{1,0}^{\oplus 10} \oplus V_{2,0}^{\oplus 12} \oplus V_{3,0}^{\oplus 6} \oplus V_{4,0}
\oplus V_{1,1}^{\oplus 9} \oplus V_{2,1}^{\oplus 8} \oplus V_{3,1}^{\oplus 3} \oplus V_{2,2}^{\oplus 2},
\end{gather*}
and, for $k\ge 4$, the multiplicities appearing in the decomposition of
$V_\lambda$ are at most $12\cdot 7^{k-4}$. Since $R^k_{G_2}$ is the
multiplicity of $V_{1,0}$ in $V^{\otimes (k-1)}$, we have
\begin{equation}\label{e.G2}
R^k_{G_2}\le 12\cdot 7^{k-5}
\end{equation}
for $k\ge 5$. Moreover, $(V^{\otimes k})^G$ is spanned by invariants given
by partitions of $\{1,\dots,k\}$ into subsets of cardinality $2$, $3$, or
$4$ by \cite[Theorem 3.23]{Schwarz}. It follows from this or \eqref{e.G2}
that $R^k\le (k-1)!$.\footnote{The sequence $R^k_{G_2}$ $(k\ge 0)$ is
sequence A059710 in the On-Line Encyclopedia of Integer Sequences. The first
terms are $1, 0, 1, 1, 4, 10, 35, 120, 455$.}

For $G=\SL_n$, we let $V_\lambda$ denote the representation of $\GL_n$
corresponding to a sequence $(\lambda_1\ge \dots \ge \lambda_n)$ (where the
$\lambda_j$'s are integers, possibly negative), so that $V\simeq
\Res^{\GL_n}_{\SL_n}V_\sigma$ and $\Res^{\GL_n}_{\SL_n}V_\lambda \simeq
\Res^{\GL_n}_{\SL_n} V_{\lambda'}$ if and only if $\lambda$ and $\lambda'$
are congruent modulo $(1,\dots,1)$. By the Littlewood-Richardson rule (or
Petri's formula), $V_\lambda\otimes V_\sigma \simeq \bigoplus V_{\lambda'}$
where $\lambda'$ runs through $\sigma$-expansions of $\lambda$ and
$V_\lambda\otimes V_\sigma^* \simeq \bigoplus V_{\lambda'}$ where $\lambda'$
runs through $\sigma$-contractions of $\lambda$. Thus $R^{k,l}\neq 0$ if and
only if $k\equiv l \pmod n$. In this case, $R^{k,l}$ is the number of
sequences of partitions $\lambda^{(0)},\dots, \lambda^{(k+l)}$, where
$\lambda^{(0)}=(0,\dots,0)$,
$\lambda^{(k+l)}=(\frac{l-k}{n},\dots,\frac{l-k}{n})$, such that for each
$0\le i<l$, $\lambda^{(i+1)}$ is a $\sigma$-expansion of $\lambda^{(i)}$,
and for each $l\le i<k+l$, $\lambda^{(i+1)}$ is a $\sigma$-contraction of
$\lambda^{(i)}$. We let $\delta(\lambda)$ denote the number of $1\le j< n$
such that $\lambda_{j+1}\neq \lambda_j$. We have $0\le \delta(\lambda)\le
n-1$. The number of $\sigma$-expansions and the number of
$\sigma$-contractions of $\lambda$ are both equal to $\delta(\lambda)+1$.
Moreover, for any $\sigma$-expansion or $\sigma$-contraction $\lambda'$ of
$\lambda$, $\lvert \delta(\lambda')- \delta(\lambda)\rvert \le 1$. Thus
$R^{k,l}\le \lfloor \frac{k+l}{2}\rfloor!\lfloor \frac{k+l-1}{2}\rfloor !$.
We will be particularly interested in $R^{k,1}$ and $R^{k,k}$. For $k\equiv
1\pmod n$, $R^{k,1}=R^{1,k}$ is the number of standard Young tableaux on the
Young diagram corresponding to
$(\frac{k-1}{n}+1,\frac{k-1}{n},\dots,\frac{k-1}{n})$, so
\[R^{k,1}_{\SL_n}=k!/\frac{(n+\frac{k-1}{n})!}{n!}\prod_{i=0}^{n-2}\frac{(i+\frac{k-1}{n})!}{i!}\]
by the hook length formula. For any $k$, $R^{k,k}$ is the dimension of
$\End(V^{\otimes k})^G$ and we have
\[R^{k,k}_{\SL_n}=\sum_\lambda m_\lambda^2\le k!,\]
where equality holds if and only if $k\le n$. Here $\lambda$ runs over
partitions $\lambda=(\lambda_1\ge \dots \ge\lambda_n\ge 0)$ satisfying
$\sum_i\lambda_i=k$, and $m_\lambda$ is the multiplicity of $V_\lambda$ in
$V_\sigma^{\otimes k}$, namely the number of standard Young tableaux on the
Young diagram corresponding to $\lambda$.
\end{remark}

\begin{remark}\label{r.R}
By the preceding remark, we have $R^{k,l}_G\le (k+l-1)!$ in all cases.
Moreover, for $G\neq G_2$, $R^{k,k}_G\le (2k-1)!!$.
\end{remark}

\begin{remark}\label{r.small}
Let us list the values of $R^{k,1}_G$ and $R^{k,k}_G$ for $1\le k\le 3$.
\begin{itemize}
\item $R^{1,1}_G=1$ in all cases.

\item $R^{2,1}_G=1$ for $G=\SO_3$ or $G_2$ and $R^{2,1}_G=0$ otherwise.

\item $R^{3,1}_G=0$ for $G=\mu_p$ ($p$ odd) or $\SL_n$,
    $R^{3,1}_{\mu_2}=1$, $R^{3,1}_{\Sp_2}=2$, $R^{3,1}_G=3$ for $G=\Sp_n$
    ($n\ge 4$) or $\SO_n$, and $R^{3,1}_{G_2}=4$.

\item $R^{2,2}_{\mu_p}=1$, $R^{2,2}_G=2$ for $G=\Sp_2$ or $\SL_n$,
    $R^{2,2}_G=3$ for $G=\Sp_n$ ($n\ge 4$) or $\SO_n$, and
    $R^{2,2}_{G_2}=4$.

\item $R^{3,3}_{\mu_p}=1$, $R^{3,3}_{\Sp_2}=5$, $R^{3,3}_{\SL_n}=6$,
    $R^{3,3}_{\Sp_4}=14$, $R^{3,3}_G=15$ for $G=\Sp_n$ ($n\ge 6$) or
    $\SO_n$, and $R^{3,3}_{G_2}=35$.
\end{itemize}
\end{remark}

\section{Moments of Jacobi sums}\label{s.3}

For subsets  $\cA_1,\dots,\cA_m$ of $\cX$, $m\ge 2$ and $n\ge 1$, we
consider the incomplete $n$-th moment of the normalized Jacobi sums
\eqref{e.seq}:
\[M^{(n)}(\cA_1,\dots,\cA_m)=\sum_{\substack{\chi_i\in \cA_i\\\chi_1\dotsm \chi_m\neq \bone}}
q^{-n(m-1)/2} J(\chi_1,\dots,\chi_m)^n.
\]
When some of the subsets are $\cX$, we adopt the following shorthand,
similar to the notation on discrepancy. We put
$M^{(n)}_k(\cA_1,\dots,\cA_m)=M^{(n)}(\cA_1,\dots,\cA_m,\cX,\dots,\cX)$ for
$m,k\ge 1$ and $M^{(n)}_k=M^{(n)}(\cX,\dots,\cX)$ for $k\ge 2$, where $\cX$
is repeated $k$ times. The statements of the following theorems make use of
the notation $R^{k,l}_{p,n}$ introduced in Lemma \ref{l.Kl}.

\begin{theorem}\label{t.moment1}
Let $m\ge 2$ and let $\cA_1,\dots,\cA_m$ be subsets of $\cX$. Let
$A_i=\#\cA_i$, $i=1,\dots,m$. Then, for $n\ge 1$,
\begin{gather}
\label{e.m11}\lvert M^{(n)}(\cA_1,\dots,\cA_m)\rvert\le (A_1A_2)^{1/2} A_3\dotsm A_m
[q+(n-1)A_2q^{1/2}]^{1/2},\\
\label{e.m12}\lvert M^{(n)}(\cA_1,\dots,\cA_m)\rvert\le A_1^{1/2}A_2^{3/4} A_3\dotsm A_m [R^{2,2}_{p,n}q^2+(n^3+R^{2,2}_{p,n}-1)q^{3/2}]^{1/4}.
\end{gather}
\end{theorem}

Recall from Remark \ref{r.R} that $R^{2,2}_{p,n}\le 3$ except for
$(p,n)=(2,7)$ in which case $R^{2,2}_{2,7}=4$.

\begin{theorem}\label{t.m2}
Let $k,m\ge 1$ and let $\cA_1,\dots,\cA_m$ be nonempty subsets of $\cX$. Let
$A_i=\#\cA_i$, $i=1,\dots,m$. Then, for $n\ge 1$,
\begin{gather}
\label{e.m21}\lvert M^{(n)}_k(\cA_1,\dots,\cA_m)\rvert
\le A_2\dotsm A_m (q-1)^kq^{-k/2}\left[A_1\lfloor n^k-\tfrac{R^{k,1}_{p,n}}{n}\rfloor+\delta R^{k,1}_{p,n}(q^{1/2}+1)\right]+T,\\
\label{e.m22}\begin{split}
&\lvert M^{(n)}_k(\cA_1,\dots,\cA_m)\rvert\\
&\qquad\le  A_2\dotsm A_m A_1^{1/2}(q-1)^{\frac{2k+1}{2}}q^{-\frac{2k+1}{4}}
\left[n^{2k+1}-1+R^{k+1,k+1}_{p,n}(q^{1/2}+1)\right]^{1/2}+T,
\end{split}
\end{gather}
where $\delta=0$ for $m=1$ and $\delta=1$ for $m\neq 1$, and
\[T=(k+1)A_1\dotsm A_m (q-1)^{k-1}q^{-n/2}.\]
\end{theorem}

\begin{theorem}\label{t.m3}
Let $k\ge 2$. Then, for $n\ge 1$,
\begin{equation}\label{e.m3}
\left\lvert M^{(n)}_k-(q-1)^kq^{\frac{1-k}{2}}R^{k,1}_{p,n}\right\rvert\le
(q-1)^kq^{-k/2} \left(\lfloor
n^k-\tfrac{R^{k,1}_{p,n}}{n}\rfloor+R^{k,1}_{p,n}\right)+[(q-1)^k-N]q^{-n/2},
\end{equation}
where $N\ge (q-2)^{k-1}(q-1-k)$ is the number of $k$-tuples
$(\rho_1,\dots,\rho_k)$, $\rho_i\in \cX$ such that $\rho_1\dotsm \rho_k\neq
\bone$.
\end{theorem}

For $k=2$ (and $n\ge 2$), we have $(q-1)^2q^{-1}n^2+(3q-5)q^{-n/2}\le n^2q$,
hence Theorem \ref{t.m3} implies the bound $\lvert M^{(n)}_2\rvert\le
n^2q+R^{2,1}_{p,n}q^{3/2}$ of Katz and Zhiyong Zheng \cite[Theorem 3]{KZ}.

As in Shparlinski \cite{Shp}, one strategy followed in the proofs consists
of applying the Cauchy-Schwarz inequality and extending the sum over
$\bar\cX$. We estimate the complete sum using Lemma \ref{l.Kl}.

Let us recall two simple facts that will be used in the proofs. The Jacobi
sums and Gauss sums are related by the formula
\[J(\chi_1,\dots,\chi_m)=G(\chi_1)\dotsm G(\chi_m)G(\chi_1\dotsm\chi_m)^{-1}
=q^{-1}G(\chi_1)\dotsm G(\chi_m)\overline{G(\chi_1\dotsm \chi_m)}.
\]
for $\chi_1,\dots,\chi_m\in \cX$ satisfying $\chi_1\dotsm \chi_m\neq \bone$.
Moreover, $G(\bone)=-1$.

\begin{proof}[Proof of Theorem \ref{t.moment1}]
We may assume $A_1\ge A_2$. Let $M^{(n)}=M^{(n)}(\cA_1,\dots,\cA_m)$. By the
facts recalled above,
\[
\begin{split}
\lvert M^{(n)}\rvert &= \left\lvert\sum_{\substack{\chi_i\in \cA_i\\\chi_1\dotsm \chi_m\neq \bone}}
\left[q^{-(m+1)/2} G(\chi_1)\dotsm G(\chi_m)\overline {G(\chi_1\dotsm \chi_m)}\right]^n\right\rvert
\le \left\lvert\sum_{\substack{\chi_i\in \cA_i\\\chi_1\dotsm \chi_m= \bone}}\right\rvert
+ \left\lvert\sum_{\chi_i\in \cA_i}\right\rvert\\
&\le A_2\dotsm A_m q^{-n/2} +W\le (A_1 A_2)^{1/2} A_3\dotsm A_m q^{-n/2}+W,
\end{split}
\]
where
\[W=\sum_{\chi_1\in \cA_1}\left\lvert \sum_{\chi_i\in \cA_i,\
i=2,\dots, m}[q^{-m/2} G(\chi_2)\dotsm G(\chi_m)\overline {G(\chi_1\dotsm \chi_m)}]^n \right\rvert.
\]
By the Cauchy-Schwarz inequality,
\[
\begin{split}
W^2 &\le A_1\sum_{\chi_1\in \cA_1}\left\lvert \sum_{\chi_i\in \cA_i,\
i=2,\dots, m}\left[q^{-m/2} G(\chi_2)\dotsm G(\chi_m)\overline {G(\chi_1\dotsm
\chi_m)}\right]^n \right\rvert^2\\
&\le A_1\sum_{\chi_1\in \Xbar}\left\lvert \sum_{\chi_i\in \cA_i,\
i=2,\dots, m}\right\rvert^2\\
&= A_1\sum_{\chi_1\in \Xbar}\sum_{\chi_i,\chi_i'\in \cA_i,\ i=2,\dots,m}\left[q^{-m}G(\chi_2)\dotsm G(\chi_m)\overline{G(\chi_1\chi_2\dotsm
\chi_m)}\overline{G(\chi'_2)\dotsm G(\chi'_m) \overline{G(\chi_1\chi'_2\dotsm \chi'_m)}}\right]^n\\
&\le A_1\sum_{\chi_i,\chi_i'\in \cA_i,\ i=2,\dots,m}q^{-n}\left\lvert\sum_{\chi_1\in \Xbar} \overline{G(\chi_1\chi_2\dotsm
\chi_m)}^n G(\chi_1\chi'_2\dotsm \chi'_m)^n\right\rvert\equalscolon X
\end{split}
\]
By \eqref{e.Fourier},
\[
\begin{split} \sum_{\chi_1\in \Xbar}
\overline{G(\chi_1\chi_2\dotsm
\chi_m)}^n G(\chi_1\chi'_2\dotsm \chi'_m)^n&= \sum_{a,b\in \F_q^\times}\sum_{\chi_1\in \bar \cX} \overline{\Kl}_n(a)\Kl_n(b)\overline{\chi_1\dotsm\chi_m}(a)\chi'_1\dotsm\chi'_m(b)\\
&=(q-1)\sum_{a\in \F_q^\times} \Kl_n(a)\Klb_n(a)\overline{\chi_2\dotsm\chi_m}\chi'_2\dotsm\chi'_m(a).
\end{split}
\]
For $\chi_2\dotsm \chi_m= \chi'_2\dotsm \chi'_m$, we have
\[\sum_{\chi_1\in \Xbar} \overline{G(\chi_1\chi_2\dotsm \chi_m)}^n
G(\chi_1\chi'_2\dotsm \chi'_m)^n= (q-2)q^n+1.
\]
Thus, by \eqref{e.Kl2} (where $R^{1,1}=1$), we have
\[
\begin{split}
X&=A_1\sum_{\substack{\chi_i,\chi_i'\in \cA_i,\ i=2,\dots,m\\\chi_2\dotsm
\chi_m= \chi'_2\dotsm \chi'_m}} +A_1\sum_{\substack{\chi_i,\chi_i'\in \cA_i,\
i=2,\dots,m\\\chi_2\dotsm \chi_m\neq \chi'_2\dotsm \chi'_m}} \\
&\le
A_1A_2(A_3\dotsm A_m)^2(q-2+q^{-n}) + A_1(A_2\dotsm A_m)^2(q-1)(n-1)q^{-1/2}\\
&= A_1A_2(A_3\dotsm A_m)^2[q -2+q^{-n}+(n-1) A_2(q-1)q^{-1/2}].
\end{split}
\]
Thus
\[\lvert M^{(n)}\rvert\le (A_1A_2)^{1/2}A_3\dotsm A_m\{q^{-n/2}+[q -2+q^{-n}+(n-1) A_2(q-1)q^{-1/2}]^{1/2}\}.\]
For \eqref{e.m11}, it suffices to show
\[q -2+q^{-n}+(n-1) A_2(q-1)q^{-1/2}\le \left\{[q+(n-1)A_2q^{1/2}]^{1/2}-q^{-n/2}\right\}^2,\]
namely
\[2q^{-n/2}[q+(n-1)A_2q^{1/2}]^{1/2}\le (n-1)A_2q^{-1/2}+2,\]
which is clear by taking squares.

It remains to show \eqref{e.m12} for $n\ge 2$. We have
\[X=A_1\sum_{\substack{\chi'_2\in \cA_2\\
\chi_i,\chi_i'\in \cA_i,\ i=3,\dots,m}} Y,
\]
where
\[Y=\sum_{\chi_2\in \cA_2}\frac{q-1}{q^n}\left\lvert\sum_{a\in \F_q^\times} \Kl_n(a)\Klb_n(a)\overline{\chi_2\dotsm\chi_m}\chi'_2\dotsm\chi'_m(a)
\right\rvert.
\]
To obtain \eqref{e.m12}, we apply the Cauchy-Schwarz inequality again:
\[
\begin{split}
Y^2&\le A_2 \left(\frac{q-1}{q^n}\right)^2\sum_{\chi_2\in \cA_2} \left\lvert\sum_{a\in \F_q^\times}\right\rvert^2
\le A_2 \left(\frac{q-1}{q^n}\right)^2\sum_{\chi_2\in \Xbar}\left\lvert\sum_{a\in \F_q^\times}\right\rvert^2\\
&= A_2 \left(\frac{q-1}{q^n}\right)^2\sum_{a,b\in \F_q^\times} \sum_{\chi_2\in \Xbar}\lvert\Kl_n(a)\Kl_n(b)\rvert^2\overline{\chi_2\dotsm\chi_m}\chi'_2\dotsm\chi'_m(ab^{-1})\\
&=A_2\frac{(q-1)^3}{q^{2n}} \sum_{a\in \F_q^\times} \Kl_n(a)^2\Klb_n(a)^2.
\end{split}
\]
Thus, by \eqref{e.Kl1}, we have
\begin{align*}
Y^2&\le A_2 (q-1)^3 [R^{2,2}q^{-1}+(n^3+R^{2,2}-1)q^{-3/2}]\\
&= A_2
[R^{2,2}q^2+(n^3+R^{2,2}-1)q^{3/2}](1-\tfrac{1}{q})^3,
\end{align*}
so that
\[X\le A_1A_2(A_3\dotsm A_m)^2 A_2^{1/2}[R^{2,2}q^2+(n^3+R^{2,2}-1)q^{3/2}]^{1/2}(1-\tfrac{1}{q})^{3/2}.\]
Therefore,
\[
\begin{split}
\lvert M^{(n)}\rvert
&\le (A_1A_2)^{1/2} A_3\dotsm A_m \left\{
q^{-n/2}+(1-\tfrac{1}{q})^{3/4}A_2^{1/4}[R^{2,2}q^2+(n^3+R^{2,2}-1)q^{3/2}]^{1/4}\right\}\\
&\le (A_1A_2)^{1/2} A_3\dotsm A_m
A_2^{1/4}[R^{2,2}q^2+(n^3+R^{2,2}-1)q^{3/2}]^{1/4}.
\end{split}
\]
Here we used the inequality $(1-\frac{1}{q})^{3/4}\le 1-\frac{3}{4q}$.
\end{proof}

\begin{proof}[Proof of Theorem \ref{t.m2}]
We have
\[
\begin{split}
&\lvert M^{(n)}_k(\cA_1,\dots,\cA_m)\rvert\\
&=\left\lvert \sum_{\substack{\chi_i\in\cA_i,\ \rho_j\in \cX\\\chi_1\dotsm \chi_m\rho_1\dotsm\rho_k\neq \bone}}
\left[q^{-(m+k+1)/2}G(\chi_1)\dotsm G(\chi_m)G(\rho_1)\dotsm G(\rho_k)\overline{G(\chi_1\dotsm\chi_m\rho_1\dotsm \rho_k)}\right]^n\right\rvert\\
&\le \left\lvert \sum_{\substack{\chi_i\in\cA_i,\ \rho_j\in \Xbar\\ \chi_1\dotsm \chi_m\rho_1\dotsm\rho_k= \bone\text{ or }\exists j, \rho_j=\bone}}\right\rvert
+\left\lvert \sum_{\chi_i\in\cA_i,\ \rho_j\in \Xbar}\right\rvert\\
&\le (k+1) A_1\dotsm A_m (q-1)^{k-1} q^{-n/2}+X ,
\end{split}
\]
where
\[X=\sum_{\chi_i\in \cA_i} q^{-n(k+1)/2} \left\lvert \sum_{\rho_j\in \Xbar}G(\rho_1)^n \dotsm G(\rho_k)^n \overline{G(\chi_1\dotsm \chi_m\rho_1\dotsm \rho_k)}^n \right\rvert.
\]
By \eqref{e.Fourier},
\[\sum_{\rho_j\in \Xbar}G(\rho_1)^n \dotsm G(\rho_k)^n \overline{G(\chi_1\dotsm \chi_m\rho_1\dotsm \rho_k)}^n
= (q-1)^k\sum_{a\in \F_q^\times}\Kl_n(a)^k\Klb_n(a)\overline{\chi_1\dotsm \chi_m} (a).
\]
Thus, by Lemma \ref{l.Kl}, we have
\[\begin{split}
X&=\sum_{\substack{\chi_i\in \cA_i\\\chi_1\dotsm\chi_m\neq \bone}} + \sum_{\substack{\chi_i\in \cA_i\\\chi_1\dotsm\chi_m= \bone}}\\
&\le (q-1)^kq^{-k/2} \left\{\sum_{\substack{\chi_i\in \cA_i\\\chi_1\dotsm\chi_m\neq \bone}} \lfloor n^k-\tfrac{R^{k,1}}{n}\rfloor
+\sum_{\substack{\chi_i\in \cA_i\\\chi_1\dotsm\chi_m= \bone}} \left[R^{k,1}q^{1/2}+(\lfloor n^k-\tfrac{R^{k,1}}{n}\rfloor +R^{k,1})\right]\right\}\\
&\le(q-1)^kq^{-k/2}\left[ A_1\dotsm A_m  \lfloor n^k-\tfrac{R^{k,1}}{n}\rfloor + \delta A_2\dotsm A_m R^{k,1}(q^{1/2}+1)\right].
\end{split}
\]

It remains to show \eqref{e.m22}. We have
\[X=\sum_{\chi_i\in \cA_i,\ i=2,\dots, m} Y,\]
where
\[Y= \sum_{\chi_1\in \cA_1}  \frac{(q-1)^k}{q^{n(k+1)/2}}\left\lvert\sum_{a\in \F_q^\times}\Kl_n(a)^k\Klb_n(a)\overline{\chi_1\dotsm \chi_m} (a)\right\rvert.\]
By the Cauchy-Schwarz inequality,
\[
\begin{split}
Y^2&\le A_1\frac{(q-1)^{2k}}{q^{n(k+1)}}\sum_{\chi_1\in \cA_1} \left\lvert\sum_{a\in \F_q^\times}\right\rvert^2
\le A_1\frac{(q-1)^{2k}}{q^{n(k+1)}}\sum_{\chi_1\in \Xbar} \left\lvert\sum_{a\in \F_q^\times}\right\rvert^2\\
&=A_1\frac{(q-1)^{2k}}{q^{n(k+1)}}\sum_{a,b\in \F_q^\times}\sum_{\chi_1\in \Xbar}\Kl_n(a)^k\Klb_n(a)\overline{\Kl_n(b)^k\Klb_n(b)}\overline{\chi_1\dotsm\chi_m}(ab^{-1})\\
&=A_1\frac{(q-1)^{2k+1}}{q^{n(k+1)}}\sum_{a\in \F_q^\times}\Kl_n(a)^{k+1}\Klb_n(a)^{k+1}.
\end{split}
\]
Thus, by \eqref{e.Kl1},
\[
\begin{split}
Y^2&\le A_1 \frac{(q-1)^{2k+1}}{q^{(k+1)}} (R^{k+1,k+1}q+(n^{2k+1}-1+R^{k+1,k+1})q^{1/2})\\
&\le A_1 \frac{(q-1)^{2k+1}}{q^{\frac{2k+1}{2}}}[n^{2k+1}-1+R^{k+1,k+1}(q^{1/2}+1)].
\end{split}
\]
Therefore,
\[X\le A_1^{1/2}A_2\dotsm A_m (q-1)^{\frac{2k+1}{2}}q^{-\frac{2k+1}{4}}[n^{2k+1}-1+R^{k+1,k+1}(q^{1/2}+1)]^{1/2}.\]
\end{proof}

\begin{proof}[Proof of Theorem \ref{t.m3}]
This is similar to the proof of \eqref{e.m21}. We have
\[
\left\lvert M^{(n)}_k-\sum_{\rho_j\in \Xbar}\right\rvert\le\left\lvert  \sum_{\substack{\rho_j\in \Xbar\\
\rho_1\dotsm\rho_k= \bone\text{ or }\exists j,
\rho_j=\bone}}\right\rvert
\le [(q-1)^k-N] q^{-n/2}.
\]
By \eqref{e.Fourier},
\[\sum_{\rho_j\in \Xbar}G(\rho_1)^n \dotsm G(\rho_k)^n \overline{G(\rho_1\dotsm \rho_k)}^n
= (q-1)^k\sum_{a\in \F_q^\times}\Kl_n(a)^k\Klb_n(a).
\]
It then suffices to apply \eqref{e.Kl1}.
\end{proof}

In Theorem \ref{t.m3}, an explicit formula for $N$ can be given by
considering the number $i$ of indices $0\le j<k$ such that the partial
product $\rho_1\dotsm \rho_j=\bone$:
\[N=\sum_{i=1}^{\lceil k/2\rceil} \binom{k-i}{i-1}(q-2)^i(q-3)^{k+1-2i}.\]

\section{Bounds for the discrepancy}\label{s.4}

The Erd\H os-Tur\'an inequality \cite[Theorem III]{ET} is a quantitative
version of Weyl's criterion on equidistribution. We will use the following
form of the inequality, due to Rivat and Tenenbaum \cite[Corollaire
1.3]{RT}.

\begin{lemma}\label{l.ET}
Let $z_1,\dots, z_N$ be complex numbers on the unit circle. Then, for any
integer $K\ge 0$, the discrepancy $D$ (Definition \ref{d.D}) satisfies
\[D\le \frac{1}{K+1}+c\sum_{n=1}^K\frac{1}{nN}\left\lvert\sum_{i=1}^N z_i^n \right\rvert,\]
where $c=0.653$.
\end{lemma}

It is shown in \cite[Theorem 1]{RT} that if $c'$ is a constant such that the
lemma holds with $c$ replaced by $c'$, then $c'\ge \frac{2}{\pi}>0.636$.

\begin{proof}[Proof of Theorem \ref{t.1}]
Let $D=D(\cA_1,\dots,\cA_m)$. The cardinality $N$ of the multiset
\eqref{e.seq} satisfies $N\ge (A_1-1)A_2\dotsm A_m$.

Since $D\le 1$ by definition, to show \eqref{e.11}, we may assume $
3A_1^{-1/3}q^{1/6}<1$, namely $A_1>3^3q^{1/2}$. As $A_1< q$, this implies
$A_1>3^{6}$. By Lemma \ref{l.ET}, for any integer $K\ge 1$, we have
\[D\le \frac{1}{K+1}+\frac{c}{(A_1-1)A_2\dotsm A_m} \sum_{n=1}^K \frac{M^{(n)}}{n}.\]
Thus, by \eqref{e.m11} and the inequality $(a+b)^{1/2}\le a^{1/2} +b^{1/2}$
for $a,b\ge 0$, we have
\[
\begin{split}
D&\le \frac{1}{K+1}+c\frac{A_1^{1/2}}{A_1-1}A_2^{-1/2}\left[\sum_{n=1}^K
n^{-1}q^{1/2}+\sum_{n=2}^K n^{-1/2}A_2^{1/2}q^{1/4}\right]\\
&\le  \frac{1}{K+1}+c(A_1 A_2)^{-1/2}\left[(1+\ln K)q^{1/2}+2(K^{1/2}-1)A_2^{1/2}q^{1/4}\right]\frac{A_1}{A_1-1},
\end{split}
\]
We choose $K$ to optimize the bound for $D$. In this optimization, we ignore
$\ln K$ as it is less sensitive to the choice of $K$. Also, we do not
attempt to optimize the coefficients. Thus we take $K=\lfloor
A_1^{1/3}q^{-1/6} \rfloor$. Then $K\ge 3$. We have $1+\ln K\le
\tfrac{1}{6}(6+\ln q).$ Thus,
\[D\le \left[(1+2c) A_1^{-1/3}q^{1/6}+\tfrac{c}{6}(A_1A_2)^{-1/2}q^{1/2}(6+\ln q)\right](1-A_1^{-1})^{-1},\]
which implies \eqref{e.11}.

To show \eqref{e.12}, we may assume $A_1\ge A_2$ and
$2A_1^{-2/7}A_2^{-1/7}q^{3/14}< 1$. Thus $2A_1^{-3/7}q^{3/14}< 1$, namely
$A_1>2^{7/3}q^{1/2}$. As $A_1<q$, this implies $A_1>2^{14/3}>25$. By Lemma
\ref{l.ET}, \eqref{e.m12}, and the case $n=1$ of \eqref{e.m11}, for any
integer $K\ge 1$,
\[
\begin{split}
D&\le \frac{1}{K+1}+c\frac{A_1^{1/2}}{A_1-1}A_2^{-1/4}\left\{\left[1+\sum_{\substack{n=2\\n\neq 7}}^K \frac{3^{1/4}}{n}+\frac{4^{1/4}}{7}\right]q^{1/2}+\sum_{n=2}^K\frac{(n^3+3)^{1/4}}{n}q^{3/8}\right\}\\
&\le\frac{1}{K+1}+cA_1^{-1/2}A_2^{-1/4}\left\{\left[1+3^{1/4}(\ln K-\tfrac{1}{7})+\tfrac{4^{1/4}}{7}\right]q^{1/2} +\tfrac{4}{3}(K-1)^{3/4}q^{3/8}\right\}\frac{A_1}{A_1-1},
\end{split}
\]
Here we used the inequality
\[\frac{(n^3+3)^{1/4}}{n}\le (n-1)^{-1/4}\]
for $n\ge 2$. Let $K=\lfloor A_1^{2/7}A_2^{1/7}q^{-3/14} \rfloor$. Then
$K\ge 2$. We have
\[1+3^{1/4}(\ln K-\tfrac{1}{7})+\tfrac{4^{1/4}}{7}\le 1+3^{1/4}(\tfrac{3}{14}\ln
q-\tfrac{1}{7})+\tfrac{4^{1/4}}{7}<
\tfrac{3^{5/4}}{14}(4+\ln q),\]
so that
\[D\le \left[(1+\tfrac{4}{3}c) A_1^{-2/7}A_2^{-1/7}q^{3/14}+ \tfrac{3^{5/4}}{14}cA_1^{-1/2}A_2^{-1/4}q^{1/2}(4+\ln q)\right](1-A_1^{-1})^{-1},\]
which implies \eqref{e.12}.
\end{proof}

\begin{proof}[Proof of Theorems \ref{t.2} and \ref{t.3}]
Let $D=D_k(\cA_1,\dots,\cA_m)$. The cardinality $N$ of the multiset
satisfies $N\ge A_1\dotsm A_m (q-2)^{k-1}(q-3)$. Let
$\epsilon=(1-\frac{2}{q})^{k-1}(1-\frac{3}{q})$.

To give a uniform treatment of the cases $m=0$ and $m\ge 1$, we adopt the
convention $A_1=\delta=1$ for $m=0$. By Lemma \ref{l.Kl}, \eqref{e.m21}, and
\eqref{e.m3}, for any integer $K\ge 1$,
\[
\begin{split}
D&\le \frac{1}{K+1} + c\epsilon^{-1}\sum_{n=1}^K n^{-1}\left[n^kq^{-k/2}+\delta R^{k,1}_{p,n}(q^{1/2}+1)A_1^{-1}q^{-k/2}+(k+1)q^{-1-n/2}\right]\\
&\le \frac{1}{K+1} + \epsilon^{-1}\frac{c}{k}[(K+1)^k-1]q^{-k/2}\\
&\qquad+\epsilon^{-1}c\left[(1+q^{-1/2})\delta(1+R'\ln K)q^{(1-k)/2}A_1^{-1}+\tfrac{1}{1-q^{-1/2}}(k+1)q^{-3/2}\right],
\end{split}
\]
where $R'=\max_{n\ge 2} R^{k,1}_{p,n} \le k!$ (Remark \ref{r.R}) and we used
the fact that $R^{k,1}_{p,1}\le 1$ for $n=1$. For $k\ge 2$, to show
\eqref{e.21} and \eqref{e.3}, we may assume
$4q^{-\frac{k}{2(k+1)}-\frac{1}{6}}\ln q<1$. This implies
$q^{2/3}>q^{\frac{k}{2(k+1)}+\frac{1}{6}}>4\ln q$, so that $q>70$. Let
$K=\lfloor q^{\frac{k}{2(k+1)}} \rfloor-1$. Then
$K+2>q^{\frac{k}{2(k+1)}}\ge q^{2/5}>5$. We have
\[
\frac{1}{K+1}=\frac{K+2}{K+1}\frac{1}{K+2}\le \frac{1}{1-q^{-2/5}}q^{-\frac{k}{2(k+1)}},\qquad (K+1)^kq^{-k/2}\le q^{-\frac{k}{2(k+1)}}.
\]
For the error terms, we have
\begin{gather*}
(1+R'\ln K)q^{(1-k)/2} \le
\tfrac{1}{2}k!(1 +\ln q) q^{-\frac{k}{2(k+1)}-\frac{1}{6}}, \\
\tfrac{c}{1-q^{-1/2}}q^{-3/2}<
q^{-\frac{k}{2(k+1)}-1}.
\end{gather*}
Therefore,
\[D\le q^{-\frac{k}{2(k+1)}}\left[\left(\frac{1}{1-q^{-2/5}}+\frac{c}{k}\right)+\tfrac{c}{2}k!q^{-1/6}(1+\ln q)(1+q^{-1/2})+(k+1)q^{-1}\right]\epsilon^{-1},\]
which implies \eqref{e.21} and \eqref{e.3}. For $k=1$, $R'=1$. To show
\eqref{e.23}, we may assume $2q^{-1/4}<1$, namely $q>16$. Let $K=\lfloor
c^{-1/2}q^{1/4}\rfloor$. Then $K\ge 2$. We have
\[1+\ln K\le 1-\tfrac{1}{2}\ln c+\tfrac{1}{4}\ln q< \tfrac{1}{4}(5+\ln q),\]
so that
\[
\begin{split}
D&\le c^{1/2}q^{-1/4} + \frac{1}{1-3q^{-1}} c^{1/2} q^{-1/4}+c \tfrac{1+q^{-1/2}}{1-3q^{-1}}\delta \tfrac{1}{4}(5+\ln q)A_1^{-1}+ 4cq^{-3/2}\\
&\le \left[\left(1+\frac{1}{1-3q^{-1}}\right)c^{1/2}+4cq^{-5/4}\right]q^{-1/4}+\tfrac{c}{4}\delta A_1^{-1}(5+\ln q)(1+2q^{-1/2}),
\end{split}
\]
which implies \eqref{e.23}.

It remains to show \eqref{e.22} and \eqref{e.24}. By Lemma \ref{l.Kl} and
\eqref{e.m22}, for any integer $K\ge 1$,
\[
\begin{split}
  D&\le \frac{1}{K+1}+c\sum_{n=1}^K n^{-1}\left\{\left[n^{k+\frac{1}{2}}+(R^{k+1,k+1}_{p,n})^{1/2}q^{1/4}(1+q^{-1/2})^{1/2}\right]A_1^{-1/2}q^{\frac{1}{4}-\frac{k}{2}}+(k+1)q^{-1-\frac{n}{2}}\right\}\epsilon^{-1}\\
  &\le \frac{1}{K+1}+ \epsilon^{-1} \frac{c}{k+\frac{1}{2}}[(K+1)^{k+\frac{1}{2}}-1]A_1^{-1/2}q^{\frac{1}{4}-\frac{k}{2}}\\
  &\qquad+\epsilon^{-1}c(1+q^{-1/2})^{1/2}\left\{1+R''^{1/2}[\ln (K+\tfrac{1}{2})-\ln\tfrac{3}{2}-\tfrac{1}{7}]+\tfrac{1}{7}(R^{k+1,k+1}_{2,7})^{1/2}\right\}A_1^{-1/2}q^{(1-k)/2}\\
  &\qquad+\epsilon^{-1}\tfrac{c}{(1-q^{-1/2})}(k+1)q^{-3/2},
\end{split}
\]
where $R''=\max_n R^{k+1,k+1}_{p,n}=(2k+1)!!$, the maximum running over all
$n\ge 2$ such that $(p,n)\neq (2,7)$, and we used the fact that
$R^{k+1,k+1}_{p,1}=1$ for $n=1$.  For $k\ge 2$, we have
$R^{k+1,k+1}_{2,7}\le 12\cdot 7^{2k-3}$ by \eqref{e.G2}. To show
\eqref{e.22}, we may assume
\[2A_1^{-\frac{1}{2k+3}}
q^{-\frac{2k-1}{2(2k+3)}-\frac{2}{7}}(7+ \sqrt{15}\ln q)<1.
\]
This implies $q^{11/14}\ge A_1^{\frac{1}{2k+3}}
q^{\frac{2k-1}{2(2k+3)}+\frac{2}{7}}> 2(7+ \sqrt{15}\ln q)$, so that $q>
150$. Let $K=\lfloor A_1^{\frac{1}{2k+3}}
q^{\frac{2k-1}{2(2k+3)}}\rfloor-1$. Then $K+2>A_1^{\frac{1}{2k+3}}
q^{\frac{2k-1}{2(2k+3)}}\ge q^{\frac{2k-1}{2(2k+3)}}\ge q^{3/14}>2$. We have
\begin{gather*}
\frac{1}{K+1}=\frac{K+2}{K+1}\frac{1}{K+2}\le \frac{1}{1-(K+2)^{-1}}A_1^{-\frac{1}{2k+3}}q^{-\frac{2k-1}{2(2k+3)}},\\
(K+1)^{k+\frac{1}{2}} A_1^{-1/2}q^{\frac{1}{4}-\frac{k}{2}}
\le A_1^{-\frac{1}{2k+3}}q^{-\frac{2k-1}{2(2k+3)}}.
\end{gather*}
For the error terms, we have
\begin{gather*}
1+R''^{1/2}[\ln (K+\tfrac{1}{2})-\ln\tfrac{3}{2}-\tfrac{1}{7}]\le \tfrac{1}{2}(2k+1)!!^{1/2}\ln q,\\
A_1^{-1/2}q^{(1-k)/2}
 \le
A_1^{-\frac{1}{2k+3}}q^{-\frac{2k-1}{2(2k+3)}-\frac{2}{7}},\\
q^{-3/2}< q^{-\frac{2k+1}{2(2k+3)}-1}
\le A_1^{-\frac{1}{2k+3}}q^{-\frac{2k-1}{2(2k+3)}-1}.
\end{gather*}
Therefore,
\[D\le A_1^{-\frac{1}{2k+3}}q^{-\frac{2k-1}{2(2k+3)}}\left[\left(\frac{1}{1-(K+2)^{-1}}+\frac{c}{k+\tfrac{1}{2}}\right)+\tfrac{1}{2}q^{-2/7}\left(7^{k-1}+(2k+1)!!^{1/2}\ln q\right)+(k+1)q^{-1}\right]\epsilon^{-1},\]
which implies \eqref{e.22}. For $k=1$, we have $R''=3$ and $R^{2,2}_{2,7}=4$
(Remark \ref{r.small}). For $A_1\ge q^{3/4}$, to show \eqref{e.24}, we may
assume $2A_1^{-\frac{1}{5}}q^{-\frac{1}{10}-\frac{1}{8}}\ln q<1$. This
implies $q^{17/40}\ge A_1^{\frac{1}{5}}q^{\frac{1}{10}+\frac{1}{8}} > 2 \ln
q$, so that $q> 300$. Let $K=\lfloor A_1^{1/5}q^{1/10}\rfloor-1$. Then $K+2>
A_1^{1/5}q^{1/10}\ge q^{1/4}>4$. We have
\[
\begin{split}
D&\le \frac{K+2}{K+1}\frac{1}{K+2}+
\frac{1}{1-3q^{-1}}\cdot\frac{2}{3}c(K+1)^{3/2}A_1^{-1/2}q^{-1/4} \\
&\qquad+\left\{1+\sqrt{3}[\ln (K+\tfrac{1}{2})-\ln\tfrac{3}{2}-\tfrac{1}{7}]+\tfrac{2}{7}\right\}A_1^{-1/2}+2q^{-3/2}
\\
&\le A_1^{-1/5}q^{-1/10}\left[\left(\frac{1}{1-q^{-1/4}}+\frac{1}{1-3q^{-1}}\cdot\frac{2}{3}c\right)+\tfrac{3}{10}\sqrt{3}q^{-1/8}(1+\ln q)+2q^{-6/5}\right],
\end{split}
\]
which implies \eqref{e.24}.
\end{proof}

\begin{proof}[Proof of Corollary \ref{c.1}]
Let $x=\log_q\#\cA_1$ and $y=\log_q\#\cA_2$. Combining the inequalities
$D\le 1$, \eqref{e.11}, and \eqref{e.12}, we get that there exists a
constant $C$ such that $D\le C q^{-f_0(x,y)}\ln q$, where
\[f_0(x,y)=\max\left\{0,\min\{\tfrac{1}{2}x+\tfrac{1}{2}y-\tfrac{1}{2},\tfrac{1}{3}x-\tfrac{1}{6}\},\min \{\tfrac{1}{2}x+\tfrac{1}{4}y-\tfrac{1}{2},\tfrac{2}{7}x+\tfrac{1}{7}y-\tfrac{3}{14}\}\right\}.\]
By symmetry, $D\le C q^{-f_0(y,x)}\ln q$, so that $D\le C q^{-f(x,y)}\ln q$,
where
\[f(x,y)=\max\{f_0(x,y),f_0(y,x)\}.\]
It is easy to check that $f(x,y)$ is as described in the corollary.
\end{proof}

\begin{proof}[Proof of Corollary \ref{c.2}]
Let $x=\log_q\#\cA_1$. For $k\ge 2$, by the inequalities \eqref{e.21} and
\eqref{e.22}, there exists a constant $C_k$ such that $D\le C_k
q^{-g_k(x)}$, where
\[g_k(x)=\max \left\{\tfrac{k}{2(k+1)},\tfrac{1}{2k+3}x+\tfrac{2k-1}{2(2k+3)}\right\}.\]
For $k=1$, by the inequalities \eqref{e.23} and \eqref{e.24}, there exists a
constant $C'$ such that $D\le C' q^{-h(x)}\ln q$, where
\[h(x)=\begin{cases}\min\{x,\frac{1}{4}\}&x\le \frac{3}{4},\\\frac{1}{5}x+\frac{1}{10}&x\ge \frac{3}{4}.\end{cases}\]
The case $k=m=1$ can be proven similarly, taking into account of the fact
that $\delta=0$ in this case.
\end{proof}

\begin{remark}
Our estimates of the moments $M_k^{(n)}$ also provide a lower bound for the
discrepancy $D_k$ for $k\ge 3$ or $p=2$. By a general result on the
discrepancy of probability measures \cite[Theorem 1]{Su}, we have
\[D_k\ge \left(\frac{2}{\pi^2}\sum_{n=1}^\infty\frac{\lvert M_k^{(n)}\rvert^2}{N^2n^2}\right)^{1/2}.\]
For $k\ge 3$, we have $R^{k,1}_{p,k-1}\ge k-1$ for $n=k-1$. Thus, by Theorem
\ref{e.m3}, we have
\[\lvert M_k^{(k-1)}\rvert\ge Nq^{\frac{1-k}{2}}(k-1)-(q-1)^kq^{-k/2}[(k-1)^k-1+k!].\]
Therefore, for $q\ge 4$, we have
\[D_k\ge \frac{\sqrt{2}}{\pi}q^{-\frac{k-1}{2}}\left[1-2q^{-1/2}(k-1)^{k-1}(\tfrac{q-1}{q-3})^k\right].\]
For $k=p=2$, we have $R^{2,1}_{2,3}=1$ for $n=3$. Thus, by Theorem
\ref{t.m3}, we have
\[\lvert M_2^{(3)}\rvert\ge (q-2)(q-3)q^{-1/2}-9(q-1)^2q^{-1}.\]
Therefore, for $q=2^f\ge 4$, we have
\[D_2\ge \frac{\sqrt{2}}{3\pi}q^{-1/2}\left[1-9q^{-1/2}(\tfrac{q-1}{q-3})^2\right].\]
\end{remark}

\subsection*{Acknowledgements}
We thank Ming Fang, Ofer Gabber, and Nicholas Katz for useful discussions.
Part of this work was done during visits of the first named and the second
named authors to l'Institut des Hautes \'Etudes Scientifiques. He thanks the
institute for hospitality and support. We thank the referee for helpful
comments.

\end{document}